\documentclass{amsart}
\usepackage{amssymb,comment,colortbl,enumitem}
\usepackage[usenames,dvipsnames]{xcolor}
\usepackage[all]{xy}
\usepackage{pstricks}

\theoremstyle{plain} 
\newtheorem{lemma}{Lemma}
\newtheorem{proposition}[lemma]{Proposition}
\newtheorem{corollary}[lemma]{Corollary}
\newtheorem*{results*}{Results}
\theoremstyle{definition}
\newtheorem*{question*}{Question}
\newtheorem{definition*}{Definition}

\newcommand{\rr}{\cellcolor{red}H}
\newcommand{\pp}{\cellcolor{gray}2}
\newcommand{\nn}{\cellcolor{green}1}
\newcommand{\vv}{\cellcolor{blue}V}
\newcommand{\nr}{\cellcolor{pink}1H}
\newcommand{\xl}{\cellcolor{pink}-V}
\newcommand{\xr}{\cellcolor{SkyBlue}-H}
\newcommand{\nl}{\cellcolor{SkyBlue}1V}
\newcommand{\np}{\cellcolor{yellow}12}
\newcommand{\group}{\mathcal{G}}

\title{An update on domineering on rectangular boards}
\author{Gabriel C. Drummond-Cole}
\thanks{This material is based in part upon work supported by the National Science Foundation under Award No. DMS-1004625.}
\begin{document}
\begin{abstract}
Domineering is a combinatorial game played on a subset of a rectangular grid between two players. Each board position can be put into one of four outcome classes based on who the winner will be if both players play optimally. In this note, we review previous work, establish the outcome classes for several dimensions of rectangular board, and restrict the outcome class in several more.
\end{abstract}
\maketitle
\section{Introduction}
Domineering, invented by G\"oran Andersson and introduced to the public by Martin Gardner~\cite{Gardner:MG}, is a game played on a rectangular grid of squares between two players. The players take turns placing dominoes on unoccupied squares of the board. Each domino covers two adjacent squares. One player, called Vertical, must place her dominoes in a ``vertical'' orientation. The other, called Horizontal, places hers in a ``horizontal'' orientation. When a player has no legal move on her turn, she loses. We shall refer to the board with vertical dimension $m$ and horizontal dimension $n$ as the $m\times n$ board or as $G_{m,n}$. 

Because this is a finite drawless turn-based perfect information game of no chance, each board position has a particular outcome class which describes the winner if neither player makes a mistake. This outcome class could be $V$ if the vertical player will win, $H$ if the horizontal player will win, $1$ if the next player to move will win, regardless of whether it is vertical or horizontal, and $2$ if the next player to move will lose. 
 
Combinatorial game theory goes further, giving each board position a value in a partially ordered Abelian group $\group$~\cite{BerlekampConwayGuy:WWFYMP}. Disjoint union of board positions corresponds to addition in the group. Outcome classes can be read off from values in the group: the outcome class of a board position is $V$ if and only if the board position is greater than $0$, is $H$ if and only if the board position is less than $0$, is $1$ if and only if the board position is incomparable to $0$, and is $2$ if and only if the board position is equal to $0$.

Domineering has been studied both by mathematicians working in combinatorial game theory and by computer scientists working in artificial intelligence. Typically the computer scientists have been concerned exclusively or primarily with outcome classes while the mathematicians have been interested in outcome classes along with other questions about the $\group$ values it takes.

Berlekamp engaged in the first systematic research into the outcome classes of rectangular boards~\cite{Berlekamp:BD}, giving precise $\group$ values for the boards $G_{2,2k+1}$.

Breuker, Uiterwijk, and van den Herik~\cite{BreukerUiterwijkvandenHerik:SEED} and Uiterwijk and van den Herik~\cite{UiterwijkvandenHerik:AI} used a computer program called DOMI that employed $\alpha$-$\beta$ pruning to determine the outcome classes of several boards. Their most recent publicly available results are in~\cite{Maastricht:W}.\footnote{They employ the convention that $V$ moves first so their $1$ is the same as our $1V$ and their $2$ the same as our $2H$. Using symmetry about the diagonal, outcome classes for the sizes they analyze can be recovered from their table.}

Lachmann, Moore, and Rapaport~\cite{LachmannMooreRapaport:WWDRB} extended this work by means of several simple rules which allowed them to combine outcome classes of smaller boards to give outcome classes for larger boards.

Bullock~\cite{Bullock:DSLCSS} wrote a computer program called Obsequi that employed $\alpha$-$\beta$ pruning to determine outcome classes of Domineering positions. This program had a number of performance enhancements over DOMI and was able to determine outcome classes for larger boards.

This paper uses a mixture of methods. We use Obsequi to analyze the disjoint union of rectangular positions with small non-rectangular positions of known $\group$-value to establish bounds in $\group$ for rectangular positions. We also employ precise $\group$ values for boards of the form $G_{2,n}$ for certain $n$, calculated with Berlekamp's formula or Siegel's cgsuite software~\cite{Siegel:C}. We investigate the implications of the methods of Lachmann et al. and apply these methods to the calculations made with Obsequi and cgsuite. We also improve one of their methods. All of this allows the following previously unpublished results:
\begin{results*}\ 
\begin{enumerate}
\item The outcome class of $G_{6,n}$ is either $1$ or $H$ for $n>29$,
\item the outcome class of $G_{8,n}$ is $H$ for $n\in\{26,30,36,40,42,46,48,50,52\}$ and all even $n>54$,
\item the outcome class of $G_{8,n}$ is either $1$ or $H$ for $n\in\{28,34,38,44,54\}$,
\item the outcome class of $G_{9,n}$ is $H$ for $n\in \{13,15,17,19,21\}$,
\item the outcome class of $G_{11,n}$ is $H$ for $n\in\{14,18\}$ and for odd $n$ greater than $31$,
\item the outcome class of $G_{15,n}$ is either $1$ or $H$ for $n\in\{6,10,14,18\}$, 
\item the outcome class of $G_{19,6}$ is $1$, and
\item the outcome class of $G_{n,2kn}$ is $H$ for all $n$ and $k$.
\end{enumerate}
\end{results*}
\section{Results on individual boards}
Lachmann et al. use a number of tools to combine results for smaller boards into results for larger boards. The simplified versions of their theorems that we will use are the following. We will use the notation $|G_{m,n}|$ to denote the $\group$ value of $G_{m,n}$.
\begin{proposition}[The one-hand-tied principle for rectangular boards]\label{prop:onehandtied}
\[|G_{m,n_1+n_2}|\le |G_{m,n_1}|+|G_{m,n_2}|.\]
The applications are as follows: if the outcome class of $G_{m,n_1}$ is $H$ and the the outcome class of $G_{m,n_2}$ is $1$ (respectively $2$ or $H$) then the outcome class of $G_{m,n_1+n_2}$ is either $1$ or $H$ (respectively $H$).
\end{proposition}
\begin{proposition}\label{prop:2m}
The outcome class of $G_{m,2km}$ is $2$ or $H$.
\end{proposition}
We shall improve this result with Proposition~\ref{prop:better2m}, but the version of Lachmann et al. suffices for the results of this section.

We can directly combine these with Bullock's results that the outcome class of $G_{6,14}$ and $G_{8,10}$ are $H$ to obtain the following proposition, which can be viewed as an application of the general principle of Proposition~\ref{prop:gcd}.
\begin{proposition}\ 
\begin{enumerate}
\item The outcome class of $G_{6,n}$ is either $1$ or $H$ for $n>29$ and
\item the outcome class of $G_{8,n}$ is $H$ for $n\in\{26,30,32,36,40,42,46,48,50,52\}$ and all even $n>54$, and is either $1$ or $H$ for $n\in\{28,34,38,44,54\}$.
\end{enumerate}
\end{proposition}
\begin{proof}
Using the one-hand-tied principle (Proposition~\ref{prop:onehandtied}), since the outcome class of $G_{6,8}$ is $H$, it suffices to prove the proposition for a representative of each residue modulo eight which is less than $38$. These are all already known for a small enough representative except the residue $5$. But by the one-hand-tied principle, since $37=14+12+11$, the outcome class of $G_{8,37}$ is $1$ or $H$.

For the second part, again using the one-hand-tied principle, since the outcome class of $G_{8,10}$ is $H$, it suffices to show the proposition for $G_{8,n}$ for $n\in \{26,28,32,34,48,64\}$. In each case, we will use the one-hand-tied principle, combining $G_{8,10}$ with $G_{8,8}$ (outcome class $1$) and/or $G_{8,16}$ (outcome class $H$ by Proposition~\ref{prop:better2m}).

Since $26=10+16$, $G_{8,26}$ has outcome class $H$. Since $28=2\times 10+8$, $G_{8,28}$ has outcome class $1$ or $H$. Similarly, $34 = 10 + 16 + 8$. The widths $32$, $48$, and $64$ are integer multiples of $16$ which suffices to show that the boards of that width have outcome class $H$. Without using the new result of Proposition~\ref{prop:better2m}, similar but slightly more intricate arguments could still show the result for $G_{8,48}$ and $G_{8,64}$.
\end{proof}
A number of boards can be analyzed by looking at the exact $\group$ values of $G_{2,n}$. The notation and definition of addition in $\group$ can be found in~\cite{BerlekampConwayGuy:WWFYMP}.
\begin{proposition}\label{prop:cgsuite}
\ \begin{enumerate}
\item The outcome classes of $G_{11,14}$ and $G_{11,18}$ are $H$, 
\item the outcome class of $G_{15,n}$ is either $1$ or $H$ for $n\in\{6,10,14,18\}$, and
\item the outcome class of $G_{19,6}$ is $1$.
\end{enumerate}
\end{proposition}
\begin{proof}
Using Berlekamp's formula or Siegel's cgsuite software, we can determine that
\[|G_{11,2}|=
\{1|||\frac{1}{2}|-1||-\frac{3}{2}|-\frac{7}{2}\}.\]
Using the one-hand-tied principle, we know that 
\[|G_{11,14}|\le 7|G_{11,2}|=\{2|0||-\frac{1}{2}|-2|||-\frac{5}{2}\}<0 \]
so $G_{11,14}$ has outcome class $H$. The one-hand-tied principle gives us the same for $G_{11,18}$.

Similarly, 
\[|G_{15,2}|=
\{3|\frac{3}{2}||1|-\frac{1}{2}|||-1\}
\]
and it is not hard to check that
\begin{eqnarray*}
3|G_{15,2}|&=&\{\frac{3}{2}|||1|-\frac{1}{2}||-1|-\frac{5}{2}\},\\
5|G_{15,2}|&=&\{\frac{7}{2}|2||\frac{3}{2}|0|||-\frac{1}{2}\},\\
7|G_{15,2}|&=&\{2|||\frac{3}{2}|0||-\frac{1}{2}|-2\}, \text{and}\\
9|G_{15,2}|&=&\{4|\frac{5}{2}||2|\frac{1}{2}|||0\}.
\end{eqnarray*}
are all incomparable with zero, so that the outcome classes of the corresponding boards are either $1$ or $H$.

Finally, 
\[|G_{19,2}|=\{\frac{3}{2}|||1|-\frac{1}{2}||-1|-\frac{5}{2}\}\]
and
\[3|G_{19,2}|=\{4|\frac{5}{2}||2|\frac{1}{2}|||0\}\]
so the outcome class of $G_{19,6}$ is either $1$ or $H$. Lachmann et al. determined that the outcome class of that board was either $1$ or $V$, so it must be $1$.
\end{proof}
The following pair of results are not particularly sharp because they rely on bounds determined by Obsequi, which is optimized for determining outcome classes quickly, not for calculating exact $\group$ values. One can establish bounds on the $\group$ values of Domineering rectangles in terms of games that can be represented by simple Domineering positions. For example, to verify that $|G_{9,7}|\le 1$, we can test whether $|G_{9,7}|+(-1)\le 0$, that is, whether the following game has outcome class either $2$ or $H$ (here gray squares are unplayable):

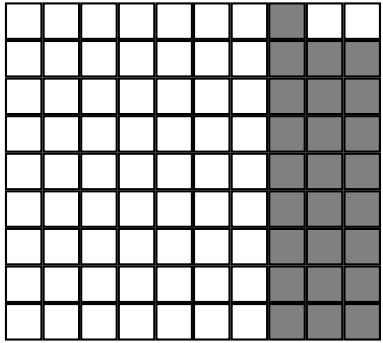
\begin{center}
\begin{pspicture}(4.5,5)
\psframe(0,0)(.5,.5)
\psframe(0,.5)(.5,1)
\psframe(0,1)(.5,1.5)
\psframe(0,1.5)(.5,2)
\psframe(0,2)(.5,2.5)
\psframe(0,2.5)(.5,3)
\psframe(0,3)(.5,3.5)
\psframe(0,3.5)(.5,4)
\psframe(0,4)(.5,4.5)
\psframe(1,0)(.5,.5)
\psframe(1,.5)(.5,1)
\psframe(1,1)(.5,1.5)
\psframe(1,1.5)(.5,2)
\psframe(1,2)(.5,2.5)
\psframe(1,2.5)(.5,3)
\psframe(1,3)(.5,3.5)
\psframe(1,3.5)(.5,4)
\psframe(1,4)(.5,4.5)
\psframe(1,0)(1.5,.5)
\psframe(1,.5)(1.5,1)
\psframe(1,1)(1.5,1.5)
\psframe(1,1.5)(1.5,2)
\psframe(1,2)(1.5,2.5)
\psframe(1,2.5)(1.5,3)
\psframe(1,3)(1.5,3.5)
\psframe(1,3.5)(1.5,4)
\psframe(1,4)(1.5,4.5)
\psframe(2,0)(1.5,.5)
\psframe(2,.5)(1.5,1)
\psframe(2,1)(1.5,1.5)
\psframe(2,1.5)(1.5,2)
\psframe(2,2)(1.5,2.5)
\psframe(2,2.5)(1.5,3)
\psframe(2,3)(1.5,3.5)
\psframe(2,3.5)(1.5,4)
\psframe(2,4)(1.5,4.5)
\psframe(2,0)(2.5,.5)
\psframe(2,.5)(2.5,1)
\psframe(2,1)(2.5,1.5)
\psframe(2,1.5)(2.5,2)
\psframe(2,2)(2.5,2.5)
\psframe(2,2.5)(2.5,3)
\psframe(2,3)(2.5,3.5)
\psframe(2,3.5)(2.5,4)
\psframe(2,4)(2.5,4.5)
\psframe(3,0)(2.5,.5)
\psframe(3,.5)(2.5,1)
\psframe(3,1)(2.5,1.5)
\psframe(3,1.5)(2.5,2)
\psframe(3,2)(2.5,2.5)
\psframe(3,2.5)(2.5,3)
\psframe(3,3)(2.5,3.5)
\psframe(3,3.5)(2.5,4)
\psframe(3,4)(2.5,4.5)
\psframe(3,0)(3.5,.5)
\psframe(3,.5)(3.5,1)
\psframe(3,1)(3.5,1.5)
\psframe(3,1.5)(3.5,2)
\psframe(3,2)(3.5,2.5)
\psframe(3,2.5)(3.5,3)
\psframe(3,3)(3.5,3.5)
\psframe(3,3.5)(3.5,4)
\psframe(3,4)(3.5,4.5)
\psframe[fillcolor=gray,fillstyle=solid](4,0)(3.5,.5)
\psframe[fillcolor=gray,fillstyle=solid](4,.5)(3.5,1)
\psframe[fillcolor=gray,fillstyle=solid](4,1)(3.5,1.5)
\psframe[fillcolor=gray,fillstyle=solid](4,1.5)(3.5,2)
\psframe[fillcolor=gray,fillstyle=solid](4,2)(3.5,2.5)
\psframe[fillcolor=gray,fillstyle=solid](4,2.5)(3.5,3)
\psframe[fillcolor=gray,fillstyle=solid](4,3)(3.5,3.5)
\psframe[fillcolor=gray,fillstyle=solid](4,3.5)(3.5,4)
\psframe[fillcolor=gray,fillstyle=solid](4,4)(3.5,4.5)
\psframe[fillcolor=gray,fillstyle=solid](4,0)(4.5,.5)
\psframe[fillcolor=gray,fillstyle=solid](4,.5)(4.5,1)
\psframe[fillcolor=gray,fillstyle=solid](4,1)(4.5,1.5)
\psframe[fillcolor=gray,fillstyle=solid](4,1.5)(4.5,2)
\psframe[fillcolor=gray,fillstyle=solid](4,2)(4.5,2.5)
\psframe[fillcolor=gray,fillstyle=solid](4,2.5)(4.5,3)
\psframe[fillcolor=gray,fillstyle=solid](4,3)(4.5,3.5)
\psframe[fillcolor=gray,fillstyle=solid](4,3.5)(4.5,4)
\psframe(4,4)(4.5,4.5)
\psframe[fillcolor=gray,fillstyle=solid](5,0)(4.5,.5)
\psframe[fillcolor=gray,fillstyle=solid](5,.5)(4.5,1)
\psframe[fillcolor=gray,fillstyle=solid](5,1)(4.5,1.5)
\psframe[fillcolor=gray,fillstyle=solid](5,1.5)(4.5,2)
\psframe[fillcolor=gray,fillstyle=solid](5,2)(4.5,2.5)
\psframe[fillcolor=gray,fillstyle=solid](5,2.5)(4.5,3)
\psframe[fillcolor=gray,fillstyle=solid](5,3)(4.5,3.5)
\psframe[fillcolor=gray,fillstyle=solid](5,3.5)(4.5,4)
\psframe(5,4)(4.5,4.5)
\end{pspicture}
\end{center}
This is precisely the sort of problem that Obsequi is equipped to handle, and it verifies that indeed, $|G_{9,7}|\le 1$.

Since this bound and the others for the following proposition are rough, it is possible that similar methods could establish the outcome class of $G_{9,11}$ (the only outstanding board of height 9) and/or boards of height eleven and odd width lower than $33$. These boards are too large for cgsuite to feasibly analyze given current computational resources.
\begin{proposition}
\ \begin{enumerate}
\item The outcome class of $G_{9,n}$ is $H$ for $n\in \{13,15,17,19,21\}$ and
\item the outcome class of $G_{11,n}$ is $H$ for odd $n$ greater than $31$.
\end{enumerate}
\end{proposition}
\begin{proof}
Since $G_{9,2}$ has outcome class $H$, for the first part it suffices to check $G_{9,13}$. As in Proposition~\ref{prop:cgsuite}, we can verify by cgsuite that $3|G_{9,2}|=\{1|-1||-\frac{3}{2}|-3\}$. Above, we described using Obsequi to verify that $|G_{9,7}|\le 1$. Then by the one-hand-tied principle, $|G_{9,13}|\le 3|G_{9,2}|+|G_{9,7}|\le \{2|0||-\frac{1}{2}|-2\}<0$, so $G_{9,13}$ has outcome class $H$. 

For the second part, it suffices to check for $G_{11,33}$ and $G_{11,35}$. we make two distinct verifications: that $|G_{11,5}|\le \frac{5}{2}$ and that $|G_{11,5}|\le \{3|2\}$. Obsequi can demonstrate these by verifying that $H$ wins if $V$ goes first on the following two positions:

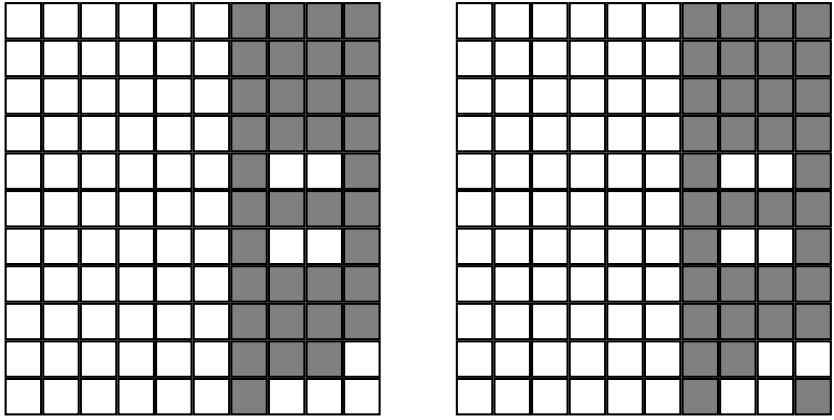
\begin{center}
\begin{pspicture}(11.5,5.5)
\psframe(0,0)(.5,.5)
\psframe(0,.5)(.5,1)
\psframe(0,1)(.5,1.5)
\psframe(0,1.5)(.5,2)
\psframe(0,2)(.5,2.5)
\psframe(0,2.5)(.5,3)
\psframe(0,3)(.5,3.5)
\psframe(0,3.5)(.5,4)
\psframe(0,4)(.5,4.5)
\psframe(0,4.5)(.5,5)
\psframe(0,5)(.5,5.5)
\psframe(1,0)(.5,.5)
\psframe(1,.5)(.5,1)
\psframe(1,1)(.5,1.5)
\psframe(1,1.5)(.5,2)
\psframe(1,2)(.5,2.5)
\psframe(1,2.5)(.5,3)
\psframe(1,3)(.5,3.5)
\psframe(1,3.5)(.5,4)
\psframe(1,4)(.5,4.5)
\psframe(1,4.5)(.5,5)
\psframe(1,5)(.5,5.5)
\psframe(1,0)(1.5,.5)
\psframe(1,.5)(1.5,1)
\psframe(1,1)(1.5,1.5)
\psframe(1,1.5)(1.5,2)
\psframe(1,2)(1.5,2.5)
\psframe(1,2.5)(1.5,3)
\psframe(1,3)(1.5,3.5)
\psframe(1,3.5)(1.5,4)
\psframe(1,4)(1.5,4.5)
\psframe(1,4.5)(1.5,5)
\psframe(1,5)(1.5,5.5)
\psframe(2,0)(1.5,.5)
\psframe(2,.5)(1.5,1)
\psframe(2,1)(1.5,1.5)
\psframe(2,1.5)(1.5,2)
\psframe(2,2)(1.5,2.5)
\psframe(2,2.5)(1.5,3)
\psframe(2,3)(1.5,3.5)
\psframe(2,3.5)(1.5,4)
\psframe(2,4)(1.5,4.5)
\psframe(2,4.5)(1.5,5)
\psframe(2,5)(1.5,5.5)
\psframe(2,0)(2.5,.5)
\psframe(2,.5)(2.5,1)
\psframe(2,1)(2.5,1.5)
\psframe(2,1.5)(2.5,2)
\psframe(2,2)(2.5,2.5)
\psframe(2,2.5)(2.5,3)
\psframe(2,3)(2.5,3.5)
\psframe(2,3.5)(2.5,4)
\psframe(2,4)(2.5,4.5)
\psframe(2,4.5)(2.5,5)
\psframe(2,5)(2.5,5.5)
\psframe(3,0)(2.5,.5)
\psframe(3,.5)(2.5,1)
\psframe(3,1)(2.5,1.5)
\psframe(3,1.5)(2.5,2)
\psframe(3,2)(2.5,2.5)
\psframe(3,2.5)(2.5,3)
\psframe(3,3)(2.5,3.5)
\psframe(3,3.5)(2.5,4)
\psframe(3,4)(2.5,4.5)
\psframe(3,4.5)(2.5,5)
\psframe(3,5)(2.5,5.5)
\psframe[fillcolor=gray, fillstyle=solid](3,0)(3.5,.5)
\psframe[fillcolor=gray, fillstyle=solid](3,.5)(3.5,1)
\psframe[fillcolor=gray, fillstyle=solid](3,1)(3.5,1.5)
\psframe[fillcolor=gray, fillstyle=solid](3,1.5)(3.5,2)
\psframe[fillcolor=gray, fillstyle=solid](3,2)(3.5,2.5)
\psframe[fillcolor=gray, fillstyle=solid](3,2.5)(3.5,3)
\psframe[fillcolor=gray, fillstyle=solid](3,3)(3.5,3.5)
\psframe[fillcolor=gray, fillstyle=solid](3,3.5)(3.5,4)
\psframe[fillcolor=gray, fillstyle=solid](3,4)(3.5,4.5)
\psframe[fillcolor=gray, fillstyle=solid](3,4.5)(3.5,5)
\psframe[fillcolor=gray, fillstyle=solid](3,5)(3.5,5.5)
\psframe(4,0)(3.5,.5)
\psframe[fillcolor=gray, fillstyle=solid](4,.5)(3.5,1)
\psframe[fillcolor=gray, fillstyle=solid](4,1)(3.5,1.5)
\psframe[fillcolor=gray, fillstyle=solid](4,1.5)(3.5,2)
\psframe(4,2)(3.5,2.5)
\psframe[fillcolor=gray, fillstyle=solid](4,2.5)(3.5,3)
\psframe(4,3)(3.5,3.5)
\psframe[fillcolor=gray, fillstyle=solid](4,3.5)(3.5,4)
\psframe[fillcolor=gray, fillstyle=solid](4,4)(3.5,4.5)
\psframe[fillcolor=gray, fillstyle=solid](4,4.5)(3.5,5)
\psframe[fillcolor=gray, fillstyle=solid](4,5)(3.5,5.5)
\psframe(4,0)(4.5,.5)
\psframe[fillcolor=gray, fillstyle=solid](4,.5)(4.5,1)
\psframe[fillcolor=gray, fillstyle=solid](4,1)(4.5,1.5)
\psframe[fillcolor=gray, fillstyle=solid](4,1.5)(4.5,2)
\psframe(4,2)(4.5,2.5)
\psframe[fillcolor=gray, fillstyle=solid](4,2.5)(4.5,3)
\psframe(4,3)(4.5,3.5)
\psframe[fillcolor=gray, fillstyle=solid](4,3.5)(4.5,4)
\psframe[fillcolor=gray, fillstyle=solid](4,4)(4.5,4.5)
\psframe[fillcolor=gray, fillstyle=solid](4,4.5)(4.5,5)
\psframe[fillcolor=gray, fillstyle=solid](4,5)(4.5,5.5)
\psframe(5,0)(4.5,.5)
\psframe(5,.5)(4.5,1)
\psframe[fillcolor=gray, fillstyle=solid](5,1)(4.5,1.5)
\psframe[fillcolor=gray, fillstyle=solid](5,1.5)(4.5,2)
\psframe[fillcolor=gray, fillstyle=solid](5,2)(4.5,2.5)
\psframe[fillcolor=gray, fillstyle=solid](5,2.5)(4.5,3)
\psframe[fillcolor=gray, fillstyle=solid](5,3)(4.5,3.5)
\psframe[fillcolor=gray, fillstyle=solid](5,3.5)(4.5,4)
\psframe[fillcolor=gray, fillstyle=solid](5,4)(4.5,4.5)
\psframe[fillcolor=gray, fillstyle=solid](5,4.5)(4.5,5)
\psframe[fillcolor=gray, fillstyle=solid](5,5)(4.5,5.5)
\put(6,0){
\psframe(0,0)(.5,.5)
\psframe(0,.5)(.5,1)
\psframe(0,1)(.5,1.5)
\psframe(0,1.5)(.5,2)
\psframe(0,2)(.5,2.5)
\psframe(0,2.5)(.5,3)
\psframe(0,3)(.5,3.5)
\psframe(0,3.5)(.5,4)
\psframe(0,4)(.5,4.5)
\psframe(0,4.5)(.5,5)
\psframe(0,5)(.5,5.5)
\psframe(1,0)(.5,.5)
\psframe(1,.5)(.5,1)
\psframe(1,1)(.5,1.5)
\psframe(1,1.5)(.5,2)
\psframe(1,2)(.5,2.5)
\psframe(1,2.5)(.5,3)
\psframe(1,3)(.5,3.5)
\psframe(1,3.5)(.5,4)
\psframe(1,4)(.5,4.5)
\psframe(1,4.5)(.5,5)
\psframe(1,5)(.5,5.5)
\psframe(1,0)(1.5,.5)
\psframe(1,.5)(1.5,1)
\psframe(1,1)(1.5,1.5)
\psframe(1,1.5)(1.5,2)
\psframe(1,2)(1.5,2.5)
\psframe(1,2.5)(1.5,3)
\psframe(1,3)(1.5,3.5)
\psframe(1,3.5)(1.5,4)
\psframe(1,4)(1.5,4.5)
\psframe(1,4.5)(1.5,5)
\psframe(1,5)(1.5,5.5)
\psframe(2,0)(1.5,.5)
\psframe(2,.5)(1.5,1)
\psframe(2,1)(1.5,1.5)
\psframe(2,1.5)(1.5,2)
\psframe(2,2)(1.5,2.5)
\psframe(2,2.5)(1.5,3)
\psframe(2,3)(1.5,3.5)
\psframe(2,3.5)(1.5,4)
\psframe(2,4)(1.5,4.5)
\psframe(2,4.5)(1.5,5)
\psframe(2,5)(1.5,5.5)
\psframe(2,0)(2.5,.5)
\psframe(2,.5)(2.5,1)
\psframe(2,1)(2.5,1.5)
\psframe(2,1.5)(2.5,2)
\psframe(2,2)(2.5,2.5)
\psframe(2,2.5)(2.5,3)
\psframe(2,3)(2.5,3.5)
\psframe(2,3.5)(2.5,4)
\psframe(2,4)(2.5,4.5)
\psframe(2,4.5)(2.5,5)
\psframe(2,5)(2.5,5.5)
\psframe(3,0)(2.5,.5)
\psframe(3,.5)(2.5,1)
\psframe(3,1)(2.5,1.5)
\psframe(3,1.5)(2.5,2)
\psframe(3,2)(2.5,2.5)
\psframe(3,2.5)(2.5,3)
\psframe(3,3)(2.5,3.5)
\psframe(3,3.5)(2.5,4)
\psframe(3,4)(2.5,4.5)
\psframe(3,4.5)(2.5,5)
\psframe(3,5)(2.5,5.5)
\psframe[fillcolor=gray, fillstyle=solid](3,0)(3.5,.5)
\psframe[fillcolor=gray, fillstyle=solid](3,.5)(3.5,1)
\psframe[fillcolor=gray, fillstyle=solid](3,1)(3.5,1.5)
\psframe[fillcolor=gray, fillstyle=solid](3,1.5)(3.5,2)
\psframe[fillcolor=gray, fillstyle=solid](3,2)(3.5,2.5)
\psframe[fillcolor=gray, fillstyle=solid](3,2.5)(3.5,3)
\psframe[fillcolor=gray, fillstyle=solid](3,3)(3.5,3.5)
\psframe[fillcolor=gray, fillstyle=solid](3,3.5)(3.5,4)
\psframe[fillcolor=gray, fillstyle=solid](3,4)(3.5,4.5)
\psframe[fillcolor=gray, fillstyle=solid](3,4.5)(3.5,5)
\psframe[fillcolor=gray, fillstyle=solid](3,5)(3.5,5.5)
\psframe(4,0)(3.5,.5)
\psframe[fillcolor=gray, fillstyle=solid](4,.5)(3.5,1)
\psframe[fillcolor=gray, fillstyle=solid](4,1)(3.5,1.5)
\psframe[fillcolor=gray, fillstyle=solid](4,1.5)(3.5,2)
\psframe(4,2)(3.5,2.5)
\psframe[fillcolor=gray, fillstyle=solid](4,2.5)(3.5,3)
\psframe(4,3)(3.5,3.5)
\psframe[fillcolor=gray, fillstyle=solid](4,3.5)(3.5,4)
\psframe[fillcolor=gray, fillstyle=solid](4,4)(3.5,4.5)
\psframe[fillcolor=gray, fillstyle=solid](4,4.5)(3.5,5)
\psframe[fillcolor=gray, fillstyle=solid](4,5)(3.5,5.5)
\psframe(4,0)(4.5,.5)
\psframe(4,.5)(4.5,1)
\psframe[fillcolor=gray, fillstyle=solid](4,1)(4.5,1.5)
\psframe[fillcolor=gray, fillstyle=solid](4,1.5)(4.5,2)
\psframe(4,2)(4.5,2.5)
\psframe[fillcolor=gray, fillstyle=solid](4,2.5)(4.5,3)
\psframe(4,3)(4.5,3.5)
\psframe[fillcolor=gray, fillstyle=solid](4,3.5)(4.5,4)
\psframe[fillcolor=gray, fillstyle=solid](4,4)(4.5,4.5)
\psframe[fillcolor=gray, fillstyle=solid](4,4.5)(4.5,5)
\psframe[fillcolor=gray, fillstyle=solid](4,5)(4.5,5.5)
\psframe[fillcolor=gray, fillstyle=solid](5,0)(4.5,.5)
\psframe(5,.5)(4.5,1)
\psframe[fillcolor=gray, fillstyle=solid](5,1)(4.5,1.5)
\psframe[fillcolor=gray, fillstyle=solid](5,1.5)(4.5,2)
\psframe[fillcolor=gray, fillstyle=solid](5,2)(4.5,2.5)
\psframe[fillcolor=gray, fillstyle=solid](5,2.5)(4.5,3)
\psframe[fillcolor=gray, fillstyle=solid](5,3)(4.5,3.5)
\psframe[fillcolor=gray, fillstyle=solid](5,3.5)(4.5,4)
\psframe[fillcolor=gray, fillstyle=solid](5,4)(4.5,4.5)
\psframe[fillcolor=gray, fillstyle=solid](5,4.5)(4.5,5)
\psframe[fillcolor=gray, fillstyle=solid](5,5)(4.5,5.5)
}
\end{pspicture}
\end{center}

Using the one-hand-tied principle we get
\[|G_{11,33}|\le 14|G_{11,2}|+|G_{11,5}|\le \{-\frac{1}{2}|-\frac{5}{2}||-3|-\frac{9}{2}\}+\frac{5}{2}<0\]
and
\[|G_{11,35}|\le 15|G_{11,2}|+|G_{11,5}|\le \{\frac{1}{2}|-\frac{3}{2}||-2|-\frac{7}{2}|||-4\}+\{3|2\}<0\]
which completes the proof.
\end{proof}
\section{Asymptotics of larger boards}
The charts of what is known for small boards, including $\group$-values where they are amenable to calculation, show some definite trends. It seems likely that for all $m$, for $N\gg m$, the outcome class of $G_{m,N}$ is $H$. $N$ may even be bounded by $m$ plus a constant. 

We cannot establish this but we can show some modest results that are best interpreted as evidence for this conjecture.

\begin{proposition}\label{prop:better2m}
For all positive $n$ and $k$, the outcome class of $G_{n,2nk}$ is $H$.
\end{proposition}
\begin{proof}
It suffices to show the result for $G_{n,2n}$.

Lachmann et al. show that this outcome class must be $2$ or $H$ using the one-hand-tied principle. Their proof is as follows. The one-hand-tied principle is valid because the horizontal player can only hurt her outcome class by refusing to move across the red line:

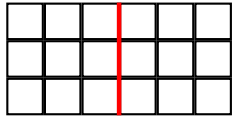
\begin{center}
\begin{pspicture}(3,1.5)
\psframe(0,0)(.5,.5)
\psframe(0,.5)(.5,1)
\psframe(0,1)(.5,1.5)
\psframe(1,0)(.5,.5)
\psframe(1,.5)(.5,1)
\psframe(1,1)(.5,1.5)
\psframe(1,0)(1.5,.5)
\psframe(1,.5)(1.5,1)
\psframe(1,1)(1.5,1.5)
\psframe(2,0)(1.5,.5)
\psframe(2,.5)(1.5,1)
\psframe(2,1)(1.5,1.5)
\psframe(2,0)(2.5,.5)
\psframe(2,.5)(2.5,1)
\psframe(2,1)(2.5,1.5)
\psframe(3,0)(2.5,.5)
\psframe(3,.5)(2.5,1)
\psframe(3,1)(2.5,1.5)
\psline[linecolor=red, linewidth=.06](1.5,0)(1.5,1.5)
\end{pspicture}
\end{center}

Because the position where the horizontal player does not move across the red line is two copies of a position which is invariant under ninety degree rotations, if the horizontal player goes second, she can copy the vertical player's moves on one half, rotating them on the other. So the position with the red line has outcome class $2$, meaning the original position has outcome class $2$ or $H$.

On the other hand, if the horizontal player goes first on the original board, she can move across the red line once at the beginning and then refuse to do so from then on, leaving the following position:

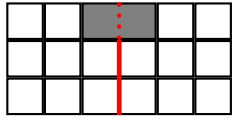
\begin{center}
\begin{pspicture}(3,1.5)
\psframe(0,0)(.5,.5)
\psframe(0,.5)(.5,1)
\psframe(0,1)(.5,1.5)
\psframe(1,0)(.5,.5)
\psframe(1,.5)(.5,1)
\psframe(1,1)(.5,1.5)
\psframe(1,0)(1.5,.5)
\psframe(1,.5)(1.5,1)
\psframe[fillcolor=gray,fillstyle=solid](1,1)(2,1.5)
\psframe(2,0)(1.5,.5)
\psframe(2,.5)(1.5,1)
\psframe(2,0)(2.5,.5)
\psframe(2,.5)(2.5,1)
\psframe(2,1)(2.5,1.5)
\psframe(3,0)(2.5,.5)
\psframe(3,.5)(2.5,1)
\psframe(3,1)(2.5,1.5)
\psline[linecolor=red, linewidth=.06](1.5,0)(1.5,1)
\psline[linecolor=red, linewidth=.06,linestyle=dotted](1.5,1)(1.5,1.5)
\end{pspicture}
\end{center}

Rotating one of the halves by ninety degrees is not literally the same as the other side, but they are the same up to horizontal and vertical reflection, which do not change the value of a position. So after the horizontal player has made her first move, she can refuse to move across the red line and copy the vertical player's moves as in the proof of Lachmann et al. This shows that the horizontal player can win going first.
\end{proof}

\begin{proposition}\label{prop:gcd}
Suppose that the outcome class of $G_{m,j}$ is $H$ and that the outcome class of $G_{m,k}$ is $H$ or $2$. Then for sufficiently high $N$, the outcome class of $G_{m, N\gcd(j,k)}$ is $H$. 
\end{proposition}
\begin{proof}
Schur's theorem says that only finitely many positive multiples of $\gcd(j,k)$ cannot be expressed as a sum of the form $aj+bk$ with nonnegative $a$ and $b$. If $N\gcd(j,k)>jk$, then any such expression can be modified so that $a$ is strictly positive. Then the one-hand-tied principle implies the result.
\end{proof}
\begin{corollary}
Suppose that the outcome class of $G_{m,j}$ is $H$ or $2$. Then for sufficiently high $N$, the outcome class of $G_{m, N\gcd(j,2m)}$ is $H$.
\end{corollary}
\begin{lemma}\label{lemma:infv}
Suppose that for a fixed height $m$, infinitely many of the boards $G_{m,n}$ have outcome class $V$. Then there exists a $k<2m$ so that for all nonnegative $i$, all boards of the form $G_{m,k+2mi}$ have outcome class $V$.
\end{lemma}
\begin{proof}
Suppose this is false. Then for each $k$ there is some $i_k$ so that $G_{m, k+2mi_k}$ can be won by the horizontal player either going first or going second. Then the same is true for $G_{m, k+2mi}$ for any $i>i_k$ and only finitely many boards of height $m$ are of outcome class $V$, a contradiction.
\end{proof}
\begin{proposition}
For a fixed height $m$, only finitely many of the boards $G_{m,n}$ have outcome class $2$.
\end{proposition}
\begin{proof}
If $G_{m,k}$ has outcome class $2$ then $G_{m,2mi+k}$ has outcome class $H$. So there can be at most $2m$ boards of height $G$ and outcome class $2$.
\end{proof}
\begin{proposition}
Let $m$ and $n$ be odd with $\gcd(m,n)=1$. Then either the set of boards of height $m$ or the set of boards of height $n$ contains only finitely many boards of outcome class $V$.
\end{proposition}
\begin{proof}
Assume that both sets contain infinitely many boards of outcome class $V$.
By Lemma~\ref{lemma:infv}, for some $k$, all boards of form $G_{m, k+2mi}$ have outcome class $V$. The same is true for boards of the form $G_{mr, k+2mi}$ by the one-hand-tied principle.

Similarly (by diagonal reflection) all boards of the form $G_{\ell+2nj, ns}$ have outcome class $H$.

Since $\gcd(m,2n)=1=\gcd(2m,n)$, there are choices of positive $i$, $j$, $r$, and $s$ such that $mr=\ell+2nj$ and $k+2mi=ns$. Then the board $G_{mr, ns}$ has outcome class both $H$ and $V$, a contradiction.
\end{proof}
\begin{corollary}
There is at most one prime $p$ such that the set of boards of height $p$, $\{G_{p,n}\}$, contains infinitely many boards of outcome class $V$.
\end{corollary}
\begin{proof}
The greatest common divisor of two primes is of course $1$. $p=2$ is already known to contain only finitely many wins for the vertical player by other means.
\end{proof}



%

%

\newpage
\appendix
\section{Table of known outcome classes}

\begin{figure}[h!]
\tabcolsep=0.04cm
\tiny
\begin{tabular}{|c||c|c|c|c|c|c|c|c|c|c|c|c|c|c|c|c|c|c|c|c|c|c|c|c|c|c|c|c|c|c|c|}
\hline
  & 1 & 2 & 3 & 4 & 5 & 6 & 7 & 8 & 9 & 10& 11& 12& 13& 14& 15& 16& 17& 18& 19& 20& 21& 22& 23& 24& 25& 26& 27&28 & 29&30&31\\
\hline\hline
 1&\pp&\rr&\rr&\rr&\rr&\rr&\rr&\rr&\rr&\rr&\rr&\rr&\rr&\rr&\rr&\rr&\rr&\rr&\rr&\rr&\rr&\rr&\rr&\rr&\rr&\rr&\rr&\rr&\rr&\rr&\rr
\\\hline
 2&\vv&\nn&\nn&\rr&\vv&\nn&\nn&\rr&\vv&\nn&\nn&\rr&\pp&\nn&\nn&\rr&\rr&\nn&\nn&\rr&\rr&\rr&\nn&\rr&\rr&\rr&\nn&\rr&\rr&\rr&\rr
\\\hline
 3&\vv&\nn&\nn&\rr&\rr&\rr&\rr&\rr&\rr&\rr&\rr&\rr&\rr&\rr&\rr&\rr&\rr&\rr&\rr&\rr&\rr&\rr&\rr&\rr&\rr&\rr&\rr&\rr&\rr&\rr&\rr
\\\hline
 4&\vv&\vv&\vv&\nn&\vv&\nn&\vv&\rr&\vv&\rr&\vv&\rr&\pp&\rr&\rr&\rr&\rr&\rr&\rr&\rr&\rr&\rr&\rr&\rr&\rr&\rr&\rr&\rr&\rr&\rr&\rr
\\\hline
 5&\vv&\rr&\vv&\rr&\pp&\rr&\rr&\rr&\rr&\rr&\rr&\rr&\rr&\rr&\rr&\rr&\rr&\rr&\rr&\rr&\rr&\rr&\rr&\rr&\rr&\rr&\rr&\rr&\rr&\rr&\rr
\\\hline
 6&\vv&\nn&\vv&\nn&\vv&\nn&\vv&\rr&\vv&\nn&\nn&\rr&\vv&\rr&\nl&\rr&   &\nr&\nn&\rr&   &\rr&\nr&\rr&\nr&\rr&\nr&\rr&   &\rr&\nr
\\\hline
 7&\vv&\nn&\vv&\rr&\vv&\rr&\nn&\rr&\rr&\rr&\rr&\rr&\rr&\rr&\rr&\rr&\rr&\rr&\rr&\rr&\rr&\rr&\rr&\rr&\rr&\rr&\rr&\rr&\rr&\rr&\rr
\\\hline
 8&\vv&\vv&\vv&\vv&\vv&\vv&\vv&\nn&\vv&\rr&\vv&   &\vv&   &   &\rr&   &\nr&   &\rr&   &   &   &\nr&   &   &   &\nr&   &\rr&
\\\hline
 9&\vv&\rr&\vv&\rr&\vv&\rr&\vv&\rr&\nn&\rr&\nr&\rr&\rr&\rr&\rr&\rr&\rr&\rr&\rr&\rr&\rr&\rr&\rr&\rr&\rr&\rr&\rr&\rr&\rr&\rr&\rr
\\\hline
10&\vv&\nn&\vv&\vv&\vv&\nn&\vv&\vv&\vv&\nn&\nl&   &\vv&   &\nl&   &   &   &   &\rr&   &\nr&   &   &   &\nr&   &   &   &\nr&
\\\hline
11&\vv&\nn&\vv&\rr&\vv&\nn&\vv&\rr&\nl&\nr&\np&\rr&\xl&\rr&\nr&\rr&\nr&\rr&\nr&\rr&\nr&\rr&\nr&\rr&\nr&\rr&\nr&\rr&\nr&\rr&\nr
\\\hline
12&\vv&\vv&\vv&\vv&\vv&\vv&\vv&   &\vv&   &\vv&\np&\vv&   &   &   &   &   &   &   &   &   &   &\rr&   &   &   &   &   & &
\\\hline
13&\vv&\pp&\vv&\pp&\vv&\rr&\vv&\rr&\vv&\rr&\xr&\rr&\np&\rr&\xl&\rr&\xl&\rr&\nr&\rr&\nr&\rr&\nr&\rr&\nr&\rr&\nr&\rr&\nr&\rr&\nr
\\\hline
14&\vv&\nn&\vv&\vv&\vv&\vv&\vv&   &\vv&   &\vv&   &\vv&\np&\nl&   &   &   &   &   &   &   &   &   &   &   &   &\rr&&\nr&
\\\hline
15&\vv&\nn&\vv&\vv&\vv&\nr&\vv&   &\vv&\nr&\nl&   &\xr&\nr&\np&   &   &\nr&   &   &   &   &   &   &   &   &   &&&\rr&
\\\hline
16&\vv&\vv&\vv&\vv&\vv&\vv&\vv&\vv&\vv&   &\vv&   &\vv&   &   &\np&   &   &   &   &   &   &   &   &   &   &   &&&&
\\\hline
17&\vv&\vv&\vv&\vv&\vv&   &\vv&   &\vv&   &\nl&   &\xr&   &   &   &\np&   &   &   &   &   &   &   &   &   &   &&&&
\\\hline
18&\vv&\nn&\vv&\vv&\vv&\nl&\vv&\nl&\vv&   &\vv&   &\vv&   &\nl&   &   &\np&   &   &   &   &   &   &   &   &   &&&&
\\\hline
19&\vv&\nn&\vv&\vv&\vv&\nn&\vv&   &\vv&   &\nl&   &\nl&   &   &   &   &   &\np&   &   &   &   &   &   &   &   &&&&
\\\hline
20&\vv&\vv&\vv&\vv&\vv&\vv&\vv&\vv&\vv&\vv&\vv&   &\vv&   &   &   &   &   &   &\np&   &   &   &   &   &   &   &&&&
\\\hline
21&\vv&\vv&\vv&\vv&\vv&   &\vv&   &\vv&   &\nl&   &\nl&   &   &   &   &   &   &   &\np&   &   &   &   &   &   &&&&
\\\hline
22&\vv&\vv&\vv&\vv&\vv&\vv&\vv&   &\vv&\nl&\vv&   &\vv&   &   &   &   &   &   &   &   &\np&   &   &   &   &   &&&&
\\\hline
23&\vv&\nn&\vv&\vv&\vv&\nl&\vv&   &\vv&   &\nl&   &\nl&   &   &   &   &   &   &   &   &   &\np&   &   &   &   &&&&
\\\hline
24&\vv&\vv&\vv&\vv&\vv&\vv&\vv&\nl&\vv&   &\vv&\vv&\vv&   &   &   &   &   &   &   &   &   &   &\np&   &   &   &&&&
\\\hline
25&\vv&\vv&\vv&\vv&\vv&\nl&\vv&   &\vv&   &\nl&   &\nl&   &   &   &   &   &   &   &   &   &   &   &\np&   &   &&&&
\\\hline
26&\vv&\vv&\vv&\vv&\vv&\vv&\vv&   &\vv&\nl&\vv&   &\vv&   &   &   &   &   &   &   &   &   &   &   &   &\np&   &&&&
\\\hline
27&\vv&\nn&\vv&\vv&\vv&\nl&\vv&   &\vv&   &\nl&   &\nl&   &   &   &   &   &   &   &   &   &   &   &   &   &\np&&&&
\\\hline
28&\vv&\vv&\vv&\vv&\vv&\vv&\vv&\nl&\vv&   &\vv&   &\vv&\vv&   &   &   &   &   &   &   &   &   &   &   &   &&\np&&&
\\\hline
29&\vv&\vv&\vv&\vv&\vv&   &\vv&   &\vv&   &\nl&   &\nl&   &   &   &   &   &   &   &   &   &   &   &   &   &&&\np&&
\\\hline
30&\vv&\vv&\vv&\vv&\vv&\vv&\vv&\vv&\vv&\nl&\vv&   &\vv&\nl&\vv&   &   &   &   &   &   &   &   &   &   &   &&&&\np&
\\\hline
31&\vv&\vv&\vv&\vv&\vv&\nl&\vv&   &\vv&   &\nl&   &\nl&   &   &   &   &   &   &   &   &   &   &   &   &   &&&&&\np
\\\hline
\end{tabular}
\end{figure}

Here a single symbol from $\{1,2,H,V\}$ designates an outcome class, a pair of symbols indicates that the outcome class must be one of the two symbols, and $-x$ indicates that the outcome class is not $x$.

For all widths greater than $31$, the boards of height $1,2,3,4,5,7,9,$ and $11$ have outcome class $H$ and the boards of height $6$ and $13$ alternate between outcome class $H$ and $1H$. For height $8$, the outcome class of even boards of width greater than $54$ is $H$ but there is some irregularity in what is known before that:

\begin{figure}[h!]\tabcolsep=0.06cm
\tabcolsep=0.04cm
\tiny
\begin{tabular}{|c||c|c|c|c|c|c|c|c|c|c|c|c|c|c|c|c|c|c|c|c|c|c|c|c|c|c|c|}
\hline
  &32 &33 &34 &35 &36 &37 &38 &39 & 40& 41& 42& 43& 44& 45& 46& 47& 48& 49& 50& 51& 52& 53&54 \\
\hline\hline
8 &\rr&   &\nr&   &\rr&   &\nr&   &\rr&   &\rr&   &\nr&   &\rr&   &\rr&   &\rr&   &\rr&   &\nr
\\\hline
\end{tabular}
\end{figure}
Of course, the same results hold with $V$ replacing $H$ for all heights greater than 31 by reflection across the diagonal.
\newpage

\section{Errata for previous work}
\subsection{Errata for Lachmann et al.}
Lachmann et al. have a table of outcome classes similar to the table above. Their table contains a few errors.

In the table, they indicate which positions' outcome classes are calculated by brute force and which outcome classes follow from applying their rules. The outcome class of the $9\times 9$ board does not follow from their rules as indicated in their table and must be calculated by brute force (the outcome class itself is correct). 

On the other hand, the outcome classes for the $2\times 27$ and $6\times 12$ boards do not need to be calculated by brute force as indicated in their table but rather follow from more intricate application of their rules. That is, the vertical player can win the $2\times 27$ board going first by partitioning it into two $2\times 13$ boards and since $27=13+14$, the one-hand-tied principle indicates that the horizontal player can also win the $2\times 27$ board going first. The $6\times 12$ board is of the form $n\times 2n$ so its outcome class must be $H$ or $2$ but using the one-hand-tied principle with $12=4+8$ indicates the outcome class must be $H$ or $1$.

There is a more serious error in the transcription of the outcome class of the $4\times 13$ board from~\cite{BreukerUiterwijkvandenHerik:SEED}. The outcome class for that board is $2$ (this was later verified by Bullock, who did not mention the discrepancy) but Lachmann et al. record it as $V$. This means that the outcome class of the $6\times 13$ board must be checked by brute force (their entry is correct). It also means that their rules do not imply that the outcome class of the $13\times 17$ board is $1$ or $H$; it may also be $2$. Furthermore, this implies by using their rules that the outcome class of the $4\times 21$ board is $H$ (this was later verified by Bullock).
\subsection{Erratum for Bullock} 
In~\cite{Bullock:W} (although not in~\cite{Bullock:DSLCSS}), Bullock records the outcome class of $G_{6,29}$ as $1H$. This does not follow from any stated rule and this case is too large for a feasible verification with Obsequi.
\bibliography{references}{}
\bibliographystyle{amsalpha}
\end{document}